\documentclass[11pt]{amsart}
\usepackage[utf8]{inputenc}

\usepackage[margin=1in]{geometry}

\usepackage{enumerate}
\usepackage{amsthm}
\usepackage{amsmath}
\usepackage{amsfonts}
\usepackage{amssymb}
\usepackage{graphicx}
\usepackage{color}
\usepackage{graphics}
\usepackage{eepic}
\usepackage{nicefrac}

\usepackage{tikzsymbols}

\newcommand{\ignore}[1]{}

\renewcommand{\Re}{\operatorname{Re}}
\renewcommand{\Im}{\operatorname{Im}}

\newcommand{\C}{{\mathbb{C}}}
\newcommand{\R}{{\mathbb{R}}}

\newcommand{\D}{{\mathbb{D}}}

\newcommand{\overbar}[1]{\mkern 1.5mu\overline{\mkern-1.5mu#1\mkern-1.5mu}\mkern 1.5mu}

\newtheorem{thm}{Theorem}[section]

\newtheorem{cor}[thm]{Corollary}

\theoremstyle{definition}

\newtheorem{example}[thm]{Example}

\theoremstyle{remark}

\title{Convexity of level lines of Martin functions and applications}
\author{Gallagher, A.-K.}
\address{Department of Mathematics, Oklahoma State University, Stillwater, OK, 74074}
\email{anne-katrin.gallagher@okstate.edu}

\author{Lebl, J.}
\address{Department of Mathematics, Oklahoma State University, Stillwater, OK, 74074}
\email{lebl@okstate.edu}

\thanks{The second author was in part supported by NSF grant DMS-1362337.}

\author{Ramachandran, K.}
\address{Department of Mathematics, Oklahoma State University, Stillwater, OK, 74074}
\email{koushik.ramachandran@okstate.edu}

\begin{document}

\begin{abstract}
Let $\Omega$ be an unbounded domain in $\mathbb{R}\times\mathbb{R}^{d}.$ A positive harmonic function $u$ on $\Omega$ that vanishes on the boundary of $\Omega$ is called a Martin function. In this note, we show that, when $\Omega$ is convex, the superlevel sets of a Martin function are also convex. As a consequence we obtain that if in addition $\Omega$ has certain symmetry with respect to the $t$-axis, and $\partial\Omega$ is sufficiently flat, then the maximum of any Martin function along a slice $\Omega\cap (\{t\}\times\mathbb{R}^d)$ is attained at $(t,0)$.
\end{abstract}

\maketitle

\section{Introduction}
This article pertains to the study of convexity of superlevel sets of positive harmonic functions in unbounded convex domains in $\mathbb{R}^n$ for $n\geq 2$. The geometry and topology of level sets of functions is a fundamental and an important topic in analysis and differential geometry. Hence it is not surprising that the study of convexity of level lines of harmonic functions has a long and rich history. Perhaps one of the earliest results in this field is that level lines of the Green function of a convex domain $\Omega\subset\mathbb{C}$ are convex, see Theorem 1.3 in \cite{AH}. Gabriel \cite{G} extended this result to convex domains in $\mathbb{R}^3$. This was further generalized by Lewis \cite{JL} who proved convexity of level lines of $p$-capacitary functions in convex rings in dimensions $n\geq 3$. Convexity of level lines to solutions of more general elliptic operators was analyzed by Caffarelli and Spruck in \cite{CS}.

\vspace{0.1in}

 Most of the results on convexity in the literature on partial differential equations are in the spirit of the following theorem, cf. \cite{CS, JL, Long}. 

\begin{thm}\label{Long}
Let $A$ and $B$ be bounded convex domains in $\mathbb{R}^n$, $n\geq 2$, with $B\subset A$. Let $u$ be a harmonic function in $A\setminus B$, continuous on $\overline{A\setminus B}$ such that $u\equiv 1$ on $\partial B$ and $u\equiv 0$ on $\partial A$. Then the set $\{x\in A\setminus B: u(x)\geq c\}\cup B$ is convex for every $c\in(0,1)$.
\end{thm}

\noindent Analogs of the above result for more general operators, including nonlinear ones, have been studied in great detail. The literature is exhaustive and we refer the reader to the references stated for instance in \cite{CS} and \cite{KO}. Broadly speaking, these results are indicative of the fact that for many elliptic operators, convexity \emph{propagates}, i.e., convexity of the boundary implies convexity of the level lines of the functions contained in the kernel of the operator in consideration. Study of curvature properties of level lines is another active area of research. These properties are not addressed in this paper. However, the reader is referred to \cite{CMY} and the references therein. 

\vspace{0.1in}

 Let $\Omega$ be an unbounded domain in $\mathbb{R}^{d+1}$, $d\geq 1$. Denote points in $\mathbb{R}^{d+1}$ by $x = (t, Y)$ for $t\in\mathbb{R}$ and $Y\in\mathbb{R}^d$. A positive harmonic function $u$ in $\Omega$ that vanishes on the boundary $\partial\Omega$ is called a Martin function on $\Omega$. For $c>0$, denote $\Gamma_c^u = \{x\in\Omega: u(x) > c\}$ the superlevel set of $u$ associated to $c$.  The superscript will be dropped whenever there is no ambiguity. We say that $\Gamma_c^u$ is strictly convex if for every $x$ such that $u(x) = c,$ $\xi^{*}H_u(x)\xi < 0$ holds for all $\xi$ in the tangent space to the level set at $x.$ Here $H_u$ denotes the Hessian matrix of $u.$ The main result of this note is:

\begin{thm}\label{convex}
Let $\Omega\subsetneq\mathbb{R}^{d+1}$ be an unbounded, convex domain, $d\geq 1$. Suppose $u$ is a Martin function on $\Omega.$ Then $\Gamma_c$ is convex for any $c>0$. Moreover, the superlevel sets $\Gamma_c$, $c>0$, are strictly convex everywhere or nowhere.  In the latter case, after a possible rotation, $\Omega = \widetilde{\Omega} \times \R^k$ for some $k \geq 1$ and some domain $\widetilde{\Omega} \subset \R^{d+1-k}$.
\end{thm}

 As a consequence of Theorem \ref{convex} we obtain that, if in addition to convexity, $\Omega$  possesses some symmetry (either rotational or reflection), then the maximum of $u$ along the slice $\Omega\cap\left(\{t\}\times\mathbb{R}^d\right)$ is attained at $(t,0).$

\begin{cor}\label{slice}
Let $f:(0,\infty)\longrightarrow\mathbb{R}^{+}$ be a Lipschitz function such that $f'(t)$ is decreasing on $(0, \infty)$ and approaches $0$ as $t\to\infty$. Suppose $D\subset\mathbb{R}^{d}$, $d\geq 1$, is a bounded, convex domain containing the origin. Moreover, suppose that $D$ is symmetric, i.e., $w\in D$ iff 
$-w\in D$. Let $$\Omega = \{(t,Y)\in\mathbb{R}^{d+1}: t>0, Y\in f(t)D\}.$$  For fixed $t>0$, set $\theta_t = \Omega\cap\left(\{t\}\times\mathbb{R}^d\right)$. Then the maximum of any Martin function $u$ on $\Omega$, when restricted to $\theta_t$, is attained at $(t,0)$. Furthermore, if $L$ is any ray in the slice $\theta_t$ from $(t,0)$ to the boundary, $\partial\theta_t$, of $\theta_t$, then $u$ is strictly decreasing along $L$ as one moves away from $(t,0)$.
\end{cor}

\noindent \textbf{Remark:} If $D$ above is taken to be the unit ball $\mathbb{B}_{1}(0)\subset\mathbb{R}^{d}$ centered at the origin, then Corollary \ref{slice} implies that $u^t(Y):=u(t, Y)$ is a radial, decreasing function of $Y$ in the ball $\mathbb{B}_{f(t)}(0)$ of radius $f(t)$.

\subsection*{Acknowledgement}
We thank Alexandre Eremenko for useful discussions and suggestions.
We are also greatly indebted to the referee for pointing out an error in an earlier version of Theorem \ref{convex}.

\section{Proofs}

\begin{proof}[Proof of Theorem \ref{convex}]
Let $G$ denote the Green function of $\Omega.$ Fix a reference point $x_0\in\Omega$. Without loss of generality it may be assumed that $u(x_0) = 1$. There exists a  sequence $\{x_n\}_{n\in\mathbb{N}}$ in $\Omega$ with no accumulation points in $\partial\Omega$ such that 
$$u(x) = \lim_{n\to\infty}\dfrac{G(x, x_n)}{G(x_0, x_n)}$$
holds, see, e.g., Section $7.1$ in \cite{Pin}, or Chapter $8$ in \cite{AG}. 
For fixed $c>0$ set $$F_n = \{x\in\Omega: G(x, x_n) > c\cdot G(x_0, x_n)\}.$$
Note that if $x\in\cup_{k=0}^{\infty}\cap_{n=k}^{\infty}F_n$, then $x\in F_n$ for all but finitely many $n\in\mathbb{N}$. Hence $u(x)\geq c,$ i.e., $x\in\overbar{\Gamma_c}.$ Moreover, if $x\in\Gamma_c$, then it follows that $x\in F_n$ for all integers $n$ sufficiently large. Thus
\begin{equation}\label{inclusion}
\overbar{\cup_{k=0}^{\infty}\cap_{n=k}^{\infty}F_n}=\overbar{\Gamma_c}.
\end{equation}
Since each $F_n$ is strictly convex (cf. \cite{Ho}, Theorem 3.3.27) and the intersection of a family of convex sets is convex, it follows that $L_k:= \cap_{n=k}^{\infty}F_n$ is convex for each $k\in\mathbb{N}$. The family $\{L_k\}_{k\in\mathbb{N}}$ is an increasing family of convex sets and hence their union $\cup_{k=0}^{\infty}\cap_{n=k}^{\infty}F_n$ is convex. Taking the closure and using equation \eqref{inclusion}, we obtain that $\overbar{\Gamma_c},$ and hence $\Gamma_c,$ is convex.
Finally, by Gabriel's theorem (Theorem 3.1.14 in \cite{Ho}), if a level set of a harmonic function is convex then it is either strictly convex everywhere or nowhere.  Moreover, if the level sets are nowhere strictly convex, there exists a $k$-dimensional subspace $T_0$, $k \geq 1$, such that $u$ is locally constant along the affine spaces $T_x$ generated by $x \in \Omega$ and $T_0$.  As $\Omega$ is convex, $\Omega \cap T_x$ is connected, and therefore $u$ is constant along each $T_x$.  Hence, after a rotation we may assume that $u$ does not depend on the last $k$ coordinates.  As $u$ is a Martin function, in particular a defining function for the domain, the domain is of the form $\widetilde{\Omega} \times \R^k$, for some $\widetilde{\Omega} \subset \R^{d+1-k}$. 
\end{proof}

\begin{proof}[Proof of Corollary \ref{slice}]
It suffices to prove that $u$ is strictly decreasing along any ray $L\subset\theta_t$ as one moves away from $(t,0)$. Note that $\Omega$ is a convex, unbounded domain under the conditions stated. Also, the assumptions on $f$ guarantee that Martin functions on $\Omega$ are unique up to a constant multiple, cf. Theorem $1.1$ in \cite{DeB}, where this is stated when $D$ is a ball. For all other domains $D$ this can easily be deduced from the proof in \cite{DeB}. The upshot of this fact is that $u(t, Y) = u(t, -Y).$ The domain $\Omega$ does not contain any one-dimensional affine subspace, and therefore it is not a product space. It now follows from Theorem \ref{convex} that the superlevel sets $\Gamma_c$, $c>0$, are strictly convex. Note that any ray can intersect a convex surface at most two times. So as one moves along $L\subset\theta_t$, starting from the center $(t,0)$, where $u$ assumes a positive value, to the boundary where $u=0$, $L$ crosses each level set exactly once -- the other crossing happens at the symmetric point. This shows that $u$ is strictly decreasing along $L$ and concludes the proof.
\end{proof}

\vspace{0.1in}

\noindent In dimension 2,  Theorem \ref{convex} and Corollary \ref{slice} may be derived using conformal mapping methods. In fact, a slightly stronger result is obtained. In the theorem below and the examples that follow, we use the complex notation $z = x+ iy.$ 

\begin{thm}
Let $\Omega \subset \{ z \in \C : \Re z > 0 \}$ be a convex domain,
which contains the positive real axis $\R_+\times\{0\}$ and is
symmetric with respect to it.  Let $u$ be a Martin function on $\Omega$.  
\begin{itemize}
  \item[(1)] Then the superlevel sets $\Gamma_c$, $c>0$, are convex. In fact, $\Gamma_c$ is strictly convex unless $\Omega$ is the right half plane.

\vspace{0.05in}

\item[(2)] For fixed $x>0$, set $\theta_{x}=\Omega\cap\{x+i\mathbb{R}\}$. Then 
$$\sup_{\theta_x} u = u(x).$$ 
Moreover, if $\Omega$ is not the right half plane, then $u(z) < u(x)$
for every $z \in \theta_x \setminus \{ x \}$, in fact, as one moves
along $\theta_x$ away from the real axis, then $u$ is strictly decreasing.
\end{itemize}
%Then for each $c>0$
%\begin{equation}
%  \sup\{u(z): \Re z \leq c,\;z \in \Omega\} \leq u(c).
%\end{equation}
%If $\Omega$ is not the right half plane we get 
%\begin{equation}
%\sup\{u(z): \Re z < c,\;z \in \Omega\} < u(c) .
%\end{equation}
\end{thm}

\begin{proof}%[Proof]
%The proof relies on $\Omega$ being biholomorphically equivalent to the unit disc as well as positive harmonic functions on the latter being well-understood. In particular, if $v$ is a positive harmonic function on the unit disc,
%continuous up to the boundary except at one point, say $i$, and zero
%on the boundary, then $v$ is a constant multiple of the Poisson Kernel, based at $i$. Thus the level sets of $v$ are
%horocycles touching the unit circle at $i$.
%
Let us first show that the boundary, $\partial\Omega$, of $\Omega$ is connected.  The domain is symmetric, so it only needs to be shown that
$\partial \Omega \cap \{x+iy : y > 0\}$ is connected.  It follows from convexity of $\overline{\Omega}$ that for
each fixed $x_0$, the set $$I_{x_0} = \partial \Omega \cap \{ x_0+iy : y > 0\}$$ is either a single point or empty.
To wit, if $I_{x_0}$ was a segment, then taking convex combinations of points of $I_{x_0}$ with points on the positive real axis, $\mathbb{R}^{+}\times\{0\},$ would yield a contradiction to convexity of $\overline{\Omega}$.
If $I_{x_0}$ is empty for some $x_0$, then it follows from the convexity of $\Omega$, that $\Omega$ contains all line segments between points on the line $x=x_0$ and points on the real positive $x$-axis. Hence $\Omega$ is the right half plane.
Therefore if the domain is not the right half plane, the boundary is a graph of a function, and is therefore
connected.

The boundary, $\partial\Omega$, is also Lipschitz as it is convex.
Let $\varphi \colon \D \to \Omega$ be a Riemann mapping.
By regularity of Riemann mappings, the function $\varphi^{-1}$ is Lipschitz up to $\partial \Omega$.
Moreover, it follows from Herglotz's theorem that if $v$ is a positive harmonic function on the unit disc $\mathbb{D}$,
that extends continuously to $\overline{\mathbb{D}}\setminus\{ i \}$ with $v =0$ on $S^1\setminus\{ i \},$ then $v$ is a constant multiple of the Poisson Kernel, based at $i$. Thus the level sets of $v$ are horocycles touching the unit circle at $i$. Hence $\Gamma_c^v$ are discs for $c>0$.
Next note that we may assume that $\varphi^{-1}$ takes $\partial \Omega$ to $S^1 \setminus \{ i \}$. Then we find the
function $u$ equals $C v \circ \varphi^{-1}$ for some constant $C$.  In particular $u$ is unique up to a multiplicative
constant.
Without loss of generality we assume that $C=1$.  In other words, the superlevel sets
of $u$ and $v$ for any $c>0$ are related by
\begin{equation*}
\varphi\left(\Gamma_c^v\right) = \Gamma_c^u.
\end{equation*}
By a theorem of Study, see pg. 273 in \cite{RR},
the image of a disc in $\D$ under $\varphi$ is convex since $\Omega$ is convex.
Thus $\Gamma_c^u$ is convex as it is the image of the disc $\Gamma_c^v$.
% This should follow from Study's theorem
%They are convex and via classical Euclidean geometry any two points
%$p,q \in v_{> a}$ the arcs of the two horocycles between $p$ and $q$
%not through the boundary of $\D$ lie entirely in $v_{> a}$.  Therefore
%we apply a theorem of Pommerenke to conclude that $u_{> a}$ is convex.

By the uniqueness of $u$ up to
a constant multiple, we find that $u$ must be invariant under conjugation,
that is $u(x+iy)=u(x-iy)$. The remainder of the proof now follows as in the proof of Corollary \ref{slice}. 
%Hence
%$u_y(x) = 0$.  Suppose $u(x_0)=c$, then $u_{>c}$ is tangent to
%the line $\Re z = x_0$ because $u_y(x_0) = 0$.  The set $u_{>c}$
%is convex so it must be on one side of the line $\Re z = x_0$.  It
%clearly cannot lie on the side $\Re z < x_0$ as this would violate the
%maximum principle.  Hence
%\begin{equation}
%\sup_{z \in \Omega \text{ and } \Re z \leq x_0 }  u(z)
%\leq u(x_0) .
%\end{equation}
%If the maximum is not a strict maximum, then we find by analytic continuation that $u$ is constant on the line $\Re z = x_0$.  As
%$u(x_0) > 0$, then we find that the entire vertical line is in $\Omega$.
%By convexity this means that $\Omega$ is the right half plane.
%
%Finally we must show that every set $u_{>c}$ intersects the positive
%real line.  This follows from symmetry, if $x+iy \in u_{>c}$
%then $x-iy \in u_{>c}$, and convexity, meaning then $x \in u_{>c}$.
\end{proof}

\vspace{0.1in}

\section{Examples} 

\begin{example}
Let us illustrate the result on a simple example.  Let $\Omega \subset \R^2$ be
$$
\Omega =
\left\{(x,y)\in\mathbb{R}^2: x > 0, -\pi/2 < y < \pi/2 \right\} .
$$
Then $u(x,y) = \sinh(x) \cos(y)$ is a Martin function on $\Omega$.
See Figure~\ref{fig:levelsetssimple} for a plot of the level sets.  Notice
that on each $\theta_x$ the function attains a strict maximum when $y=0$.
\begin{figure}[h!t]
\begin{center}
\includegraphics{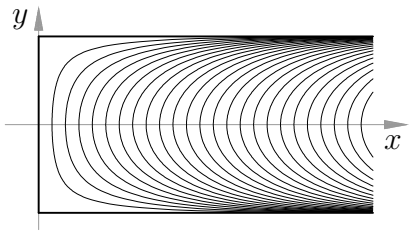}
\end{center}
\caption{Level sets of the Martin function
$u(x,y) = \sinh(x) \cos(y)$.\label{fig:levelsetssimple}}
\end{figure}
\end{example}

\begin{example}
In contrast to Theorem \ref{convex}, it is possible to find a non-convex domain and a Martin function for which \emph{every} superlevel set $\Gamma_c$ is non-convex. To illustrate this, let consider $\Omega\subset\mathbb{R}^2$ defined by $$\Omega= \left\{(x,y)\in\mathbb{R}^2: x > 0, \|(x,y)\|>1\right\}.$$
Clearly $\Omega$ is non-convex. For $(x,y)\in\Omega$, define $$u(x,y) = x- \dfrac{x}{x^2 + y^2}.$$ 
It is easy to check that $u$ is a Martin function on $\Omega$.
Observe that $u(x,y) = u(x,-y)$, i.e., every level set 
$\{(x,y)\in\Omega: u(x,y) = c\}$ crosses the $x$-axis. Next note that 
$u(x, y) = u(x, -y) > u(x,0)$ for all $y\neq 0$ and $x>1$. It now follows that every superlevel set $\Gamma_c$ is non-convex. In fact, fix an $x_0>1$ and consider the set $\Gamma = \{(x,y)\in\mathbb{R}^2: u(x,y) > u(x_0, 0)\}$. The first observation guarantees that every superlevel set of $u$ can be written in this form for some $x_0 >1$. It follows from the second observation that both
$(x_0,y)$ and $(x_0,-y)$ belong to $\Gamma$ for $y\neq 0$, but their midpoint $(x_0, 0)$ does not.

\vspace{0.1in}

In view of the previous example, it is natural to ask whether non-convexity also propagates. That is, if we start with a non-convex domain, do the level sets of a Martin function remain non-convex? Our next example shows that this need not be the case.
\end{example}

\begin{example}
Let $S= \{z\in\mathbb{C}: \Re(z)>0, \hspace{0.05in} |\Im(z)|< \Re(z)\}$ and set $\Omega = S\setminus[0, 1].$ Then $\Omega$ is non-convex. By mapping $\Omega$ conformally onto the right half plane, we obtain that
$$u(z) = \Re\sqrt{z^4 - 1}$$
is a Martin function on $\Omega$. It is clear that if $\epsilon >0$ is sufficiently small, then the superlevel set $\Gamma_{\epsilon}$ is non-convex, see Figure~\ref{fig:levelsets}. We show that if $c$ is sufficiently large,  then $\Gamma_{c}$ is convex. This is done by comparing $u$ with $v(z) = \Re(z^2)$. Note that $v$ is a Martin function on $S$. By Theorem \ref{convex}, it follows that the level sets of $v$ are strictly convex. Since $v-u$ and its derivatives up to second order vanish at infinity, it is expected that $\Gamma_c^u$ is also convex for $c$ sufficiently large. 

\begin{figure}[h!t]
\begin{center}
\includegraphics{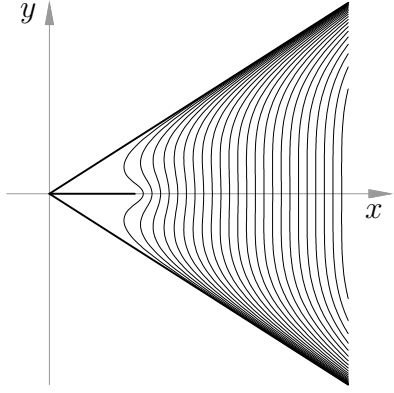}
\end{center}
\caption{Level sets of the Martin function $u(z) = \Re \sqrt{z^4-1}$.\label{fig:levelsets}}
\end{figure}

\vspace{0.1in}

\noindent To wit, let $z_0\in\Omega$ be a point with $u(z_0) = c$ for some large, fixed $c$ (to be chosen later). This forces  $v(z_0)\geq\sqrt{c}.$ Write $H_u$ and $H_v$ for the real Hessian matrix of $u$ resp. $v$. Then $$H_u=
\begin{bmatrix}
u_{xx} & u_{xy}\\
u_{yx}& u_{yy}\\
\end{bmatrix}
  \;\;\text{and}\;\; H_v=
\begin{bmatrix}
2 & 0\\
0&-2\\
\end{bmatrix}.
$$
 Denote by $A^{*}$ the transpose of a matrix $A$ and by $||A||$ the operator norm of $A.$ Set $T_u=(u_y,-u_x)$ and $T_v=(v_y,-v_x)$. Note that $T_u(z_0)$ is tangential to $\{z\in\Omega: u(z)=u(z_0)\}$ at $z_0$, and in fact spans the tangent space there; analogous statements hold for $T_v$. Moreover,  $T_v(z) = 2(y,x)$. It is then easy to see that 
\begin{equation}\label{hessv}
\left(T_v^{*}H_vT_v\right) (z) = -8v(z) < 0, \hspace{0.07in} 
\left\|(H_vT_v)(z)\right\|= \mathcal{O}\left(|z|\right).
\end{equation}
We will show that
\begin{equation}\label{hessu}
\left(T_u^{*}H_u T_u\right) (z_0) = \left(T_v^{*}H_v T_v\right) (z_0) + \mathcal{O}\left(|z_0|^{-2}\right) = -8v(z_0) + \mathcal{O}\left(|z_0|^{-2}\right) < 0
\end{equation}
as this implies the strict convexity of $\Gamma_{c}$ for $c$ chosen sufficiently large.

To prove \eqref{hessu}, set $f(z) = z^2 - \sqrt{z^4 - 1}$ and notice that $f$ is holomorphic in $\Omega$. Furthermore,  for large values of $|z|$, it follows that $$|f''(z)| = \mathcal{O}\left(|z|^{-4}\right)\text{ and } |f'(z)| = \mathcal{O}\left(|z|^{-3}\right).$$ Since $\Re(f(z)) = v(z) - u(z)$ it now follows that
\begin{equation}\label{bigO}
\left\|(H_u - H_v)(z)\right\|= \mathcal{O}\left(|z|^{-4}\right), \hspace{0.05in} \left||(T_u - T_v)(z)\right|| = \mathcal{O}\left(|z|^{-3}\right).
\end{equation}
Therefore we have
\begin{align}\label{est1}
(T_u^{*}H_uT_u - T_v^{*}H_vT_v)(z) 
  &= \left(T_u^{*}\left(H_uT_u - H_vT_v\right)\right)(z) + \left(\left(T_u^{*} 
  - T_v^{*}\right)H_vT_v\right)(z) \\
  &= \left(T_u^{*}\left(H_uT_u - H_vT_v\right)\right)(z) 
  + \mathcal{O}\left(|z|^{-3}\right)|z|\notag\\
  &= \left(T_u^{*}\left(H_uT_u - H_vT_v\right)\right)(z) + \mathcal{O}\left(|z|^{-2}\right).\notag
\end{align}
Here equations \eqref{hessv} and \eqref{bigO} were used to bound the second term on the right hand side in \eqref{est1}. We then write 
\begin{align}\label{est2}
\left(T_u^{*}\left(H_uT_u - H_vT_v\right)\right)(z)&= \left(T_u^{*}\left(H_u-H_v\right)T_u\right)(z) - \left(T_u^{*}H_v\left(T_v - T_u\right)\right(z).
\end{align}
In \eqref{est2}, the first term on the right can be estimated by equation \eqref{bigO} using the fact that $||T_u(z)|| = \mathcal{O}\left(|z|\right)$. This implies that $$\left\|\left(T_u^{*}\left(H_u-H_v\right)T_u\right)(z)\right\| = \mathcal{O}\left(|z|^{-4}\right)|z|^2 = \mathcal{O}\left(|z|^{-2}\right).$$ 
Similarly we obtain 
$$||\left(T_u^{*}H_v\left(T_v - T_u\right)\right)(z)||=
\mathcal{O}\left(|z|^{-2}\right)$$ 
using equations \eqref{hessv} and \eqref{bigO}. These estimates  prove \eqref{hessu}. Hence the strict convexity of $\Gamma_c$ holds for sufficiently large $c$. 
\end{example}

\section{Maximum on a slice}
In Section $2$, it was shown that, when $\Omega$ is convex and
symmetric, the maximum of any Martin function on $\Omega$, when restricted to a slice $\theta_t$, occurs at $(t,0)$. In this section, it is shown that this phenomenon is also valid for rotationally symmetric domains whose boundaries are sufficiently flat, without imposing any convexity assumptions. To be more specific, let 
$$\Omega = \left\{(t, Y)\in\mathbb{R}^{d+1}: t>0, |Y| < f(t)\right\},$$
where $f$ is a positive Lipschitz function defined on $(0,\infty)$ such that $f'(t)\to 0$ as $t\to\infty$. Then $\Omega$ is an unbounded domain that is  rotationally invariant in the $Y$-direction. Suppose that $u$ is a Martin function on $\Omega$. For fixed $t>0$, we consider the function $u^t(Y) = u(t, Y)$ and prove that the maximum of $u$ on the slice $\theta_t = \Omega\cap\left(\{t\}\times\mathbb{R}^d\right)$ occurs at $(t, 0)$ for $t$ sufficiently large. 

\begin{thm}
Let $\Omega$, $u$ and $u^t$ be as in the previous paragraph. Then there exists a $t_0>0$ such that  $$\max_{\theta_t}u = u^t(0) = u(t,0)$$
holds for all $t>t_0$.
\end{thm}

\textbf{Remark:} When $\Omega =\{(t, Y)\in\mathbb{R}^{d+1}: |Y| < 1\} = \mathbb{R}\times \mathbb{B}_1(0)$ is a cylinder, then every Martin function $v$ has the form $v(t, Y) = (Ae^{\sqrt{\lambda}t} + Be^{-\sqrt{\lambda}t})\phi(Y),$ see \cite{MI}. Here $\lambda$ is the principal eigenvalue of $\Delta_{d}$ on $\mathbb{B}_1(0)$ and $\phi$ is the corresponding eigenfunction. In this case, it is trivial to see that for a fixed $t$, the restriction to a slice yields a constant multiple of $\phi$. Since $\phi$ is radially decreasing, the claim follows. 

\begin{proof}
We shall show that for $t$ sufficiently large, the restriction of $u$ to the slice $\theta_t$ yields a superharmonic function. This fact is then used to deduce that the maximum occurs on the $t$-axis. 

\vspace{0.1in}

The basic idea  is to first rescale $\Omega$ to make it close to a cylinder. Then we infer the behaviour of $u$ from the solution on the cylinder, see remark above. This idea has been used in \cite{KR}. In the following, we reproduce parts of the proof from \cite{KR} for the benefit of the reader.  Let $\{s_n\}_n$ be a real positive sequence  that tends to infinity. Write $e_1 = (1, 0, 0,...,0)$. Define $$\Omega_n = \{(t, Y)\in\Omega: s_n/2 < t < 3s_n/2\}\;\text{ and }\; S_n = \dfrac{\Omega_n - s_ne_1}{f(s_n)}.$$ It is easy to verify that 
$$S_n = \left\{(t,Y)\in\mathbb{R}^{d+1}: -s_n/2 < t < s_n/2, \hspace{0.1in} |Y| < \dfrac{f(s_n + tf(s_n))}{f(t_n)}\right\}.$$  
We claim that for every compact $K$, $d_H(S_n\cap K, \mathcal{C}\cap K)\to 0$ as $n\to\infty,$ where $\mathcal{C}$ is the unit cylinder $\mathbb{R}\times \mathbb{B}_1(0)$ and $d_H$ denotes the Hausdorff distance. Indeed, by the mean value theorem and the assumptions on $f$, there exists $\widetilde{s_n}$ in $[s_n, s_n + tf(s_n)]$ such that
$$\left|\dfrac{f(s_n + tf(s_n))}{f(s_n)} - 1\right| = \left|\dfrac{tf(s_n)f'(\widetilde {s_n})}{f(s_n)}\right|\to 0$$ 
\noindent  uniformly as $n\to\infty,$ provided $t$ stays in a compact set. Let $M(t) = \sup_{\theta_t}u.$ Now define a function $v_n$ on $S_n$ by 
$$v_n(\xi) = \dfrac{u(f(s_n)\xi + s_ne_1)}{M(s_n)}.$$
\noindent For each $n$, $v_n$ is a positive harmonic function on $S_n,$ $$v_n(0) = \dfrac{u(s_ne_1)}{M(s_n)}\leq 1$$ and $v_n$ vanishes on the lateral boundary of $S_n$ (inherited from $u$). Therefore one can find a subsequence, which we still call $ v_n,$ such that $v_{n}\to v$ uniformly on compact sets of $\mathcal{C}$, cf. Lemma 1 in \cite{KR}. Then $v$ is a positive harmonic function on the cylinder $\mathcal{C}$. Using a uniform Boundary Harnack principle, we can show that $v$ vanishes on the boundary $\partial\mathcal{C}$, cf. Lemma 2 in \cite{KR}. Therefore,
$$ \dfrac{u(f(s_n)\xi + s_ne_1)}{M(s_n)}\to v(\xi)$$ 
uniformly for $\xi$ on compact subsets of $\mathcal{C}$, even up to the boundary. Since this is true for any sequence $\{s_n\}$ tending to infinity, it follows that
$$\lim_{s\to\infty} v_s(\xi)=\lim_{s\to\infty}\dfrac{u(f(s)\xi + se_1)}{M(s)}= v(\xi),$$
where the convergence is uniform for $\xi$ on compact subsets of $\mathcal{C}$. This implies that the corresponding second partial derivatives in the first coordinate also converge uniformly on compacts. Recall that every Martin function $v$ on the cylinder has the form $v(t, Y) = (Ae^{\sqrt{\lambda}t} + Be^{-\sqrt{\lambda}t})\phi(Y).$ %Here $\lambda$ is the principal eigenvalue of $\Delta_{d-1}$ on $B$ and $\phi$ is the corresponding eigenfunction. 
A straightforward computation yields $\partial_{tt}v(\xi) > 0$
for all $\xi$. Since $v_s\to v,$ for all sufficiently large $s$ we have
$$\partial_{tt}v_s(\xi) >0\;\;\text{for}\;\;\xi = (s, Y)\in\mathcal{C}.$$
Observe that if  $\xi = (0,Y)\in\mathcal{C}$,  then $f(s)\xi + se_1 = (s, f(s)Y) \in\theta_s$. Taking second partials in the first coordinate and substituting $\xi=(0,Y),$ we obtain $$\partial_{tt}v_s(\xi) = \dfrac{f(s)^2}{M(s)}\partial_{tt}u(s,f(s)Y) > 0.$$ In particular, $\partial_{tt}u(t, Y) > 0 $ on $\theta_t$ for all sufficiently large $t$. This means $\Delta_{d}u^t = - \partial_{tt}u < 0.$ Hence $u^t$ is superharmonic on the slice $\theta_t$. 

\vspace{0.1in}

\noindent The assumptions on $f$ imply that Martin functions on $\Omega$ are unique up to a constant multiple, cf. Theorem 1.1  in \cite{DeB}. As before, this forces $u(t, Y) = u(t, TY)$ for any orthogonal matrix $T\in \text{O}(d).$ In other words, $u^t(Y) = u^t(TY)$. This means that $u^t$ is a superharmonic function that is radial. A simple use of the super mean value property for superharmonic functions now yields that the maximum of $u^t$ is attained at the center of the ball, i.e., at $(t,0)$. 
\end{proof}

\noindent The above result raises a natural question: 

\vspace{0.1in} 

\noindent \textbf{Question:} Let $\Omega$ be an unbounded domain in $\mathbb{R}^{d+1}$ that contains the positive $t$-axis and is rotationally invariant in the $Y$-direction, that is $(t, Y)\in\Omega$ if and only if $(t, TY)\in\Omega$ for any orthogonal matrix $T\in \text{O}(d)$. What are necessary and sufficient geometric conditions on $\Omega$ so that if $u$ is \emph{any} Martin function on $\Omega$, then $\max_{\theta_t}u = u(t,0)$ holds for sufficiently large $t$?

\end{document}